\documentclass[12pt]{iopart}

\usepackage{iopams} 

\expandafter\let\csname equation*\endcsname\relax
\expandafter\let\csname endequation*\endcsname\relax
\usepackage{amsmath}
\usepackage{amsfonts}
\usepackage{amssymb}
\usepackage{amsthm}
\usepackage{mathptmx}
\usepackage{mathrsfs}
\usepackage{latexsym}
\usepackage{times}
\usepackage[mathscr]{euscript}
\usepackage[isolatin]{inputenc}

\usepackage[pagebackref]{hyperref}

\DeclareMathAlphabet{\mathpzc}{OT1}{pzc}{m}{it}

\newtheorem{thm}{Theorem}[section]
\newtheorem{lem}[thm]{Lemma}
\newtheorem{prop}[thm]{Proposition}

\theoremstyle{definition}
\newtheorem{defn}[thm]{Definition}
\newtheorem{ex}[thm]{Example}
\newtheorem*{aknow}{Acknowledgments}

\theoremstyle{remark}
\newtheorem{rem}[thm]{Remark}

\usepackage{color}

\newcommand{\spk}{\mathfrak{sp}}

\newcommand{\slk}{\mathfrak{sl}}

\newcommand{\ok}{\mathfrak{o}}
\newcommand{\ospk}{\mathfrak{osp}}

\newcommand{\g}{\mathfrak{g}}

\newcommand{\hk}{\mathfrak{h}}

\newcommand{\sk}{\mathfrak{s}}

\newcommand{\qk}{\mathfrak{q}}

\newcommand{\rk}{\mathfrak{r}}

\newcommand\CC{\mathbb C}

\newcommand\ZZ{\mathbb Z}

\newcommand\PP{\mathbb P}

\newcommand{\Zs}{\mathscr Z}

\newcommand{\ad}{\rm{ad}}

\newcommand{\End}{\rm{End}}
\newcommand{\Der}{\rm{Der}}

\renewcommand\hat\widehat
\renewcommand\tilde\widetilde 
\newcommand{\spa}{\rm{span}}

\newcommand{\OO}{\operatorname{O}}
\newcommand{\Sp}{\operatorname{Sp}}

\newcommand{\zero}{\overline{0}}

\newcommand{\ze}{{\scriptscriptstyle{\overline{0}}}}
\newcommand{\un}{{\scriptscriptstyle{\overline{1}}}}

\newcommand{\gb}{\overline{\g}}

\newcommand{\oplusp}{{ \ \stackrel{\bot}{\oplus} \ }}

\newcommand{\ps}{\PP^1}

\begin{document}
\title[Quadratic and odd quadratic Lie superalgebras in low dimensions]{A classification of quadratic and odd quadratic Lie superalgebras in low dimensions}

\author{Minh Thanh Duong}

\address{Department of Physics, Ho Chi Minh city University of Pedagogy, 280 An Duong Vuong, Ho Chi Minh city, Vietnam.}

\ead{thanhdmi@hcmup.edu.vn}
\begin{abstract}
  In this paper, we give an expansion of two notions of double extension and $T^*$-extension for quadratic and odd quadratic Lie superalgebras. Also, we provide a classification of quadratic and odd quadratic Lie superalgebras
 up to dimension 6. This classification is considered up to isometric isomorphism, mainly in the solvable case and the obtained Lie superalgebras are indecomposable.
\end{abstract}
\maketitle
\normalsize

\section{Introduction}
\hspace{0.8cm}Throughout the paper, all considered vector spaces are finite-dimensional complex vector spaces.

We recall a well-known example in Lie theory as follows. Let $\g$ be a complex Lie algebra and $\g^*$ its dual space. Denote by $\ad: \g\rightarrow\End(\g)$ the adjoint representation and by $\ad^*:\g\rightarrow\End(\g^*)$ the coadjoint representation of $\g$. 
The semidirect product $\gb=\g\oplus\g^*$ of $\g$ and $\g^*$ by the representation $\ad^*$ is a Lie algebra with the bracket given by:
$$
[X+f,Y+g] = [X,Y]+{\ad}^*(X)(g)-{\ad}^*(Y)(f),\ \forall X,Y\in\g,\ f,g\in\g^*.
$$

Remark that $\gb$ is also a {\em quadratic} Lie algebra with invariant symmetric bilinear form $B$ defined by:
$$
B(X+f,Y+g)=f(Y)+g(X),\ \forall X,Y\in\g,\ f,g\in\g^*.
$$

In the way of how to generalize this example, A. Medina and P. Revoy gave the notion of double extension to completely characterize all quadratic Lie algebras \cite{MR85}. Another generalization is called $T^*$-extension given by M. Bordemann that is sufficient to describe solvable quadratic Lie algebras \cite{Bor97}. In this paper, we shall just give an expansion of these two notions for Lie superalgebras. In particular, we present a way to obtain a quadratic Lie superalgebra from a Lie algebra and a symplectic vector space. It is regarded as a rather special case of the notion of generalized double extension in \cite{BBB}. In a slight change of the notion of $T^*$-extension, we give a manner of how to get an odd quadratic Lie superalgebra from a Lie algebra. Odd quadratic Lie superalgebras were studied in \cite{ABB10}. For the definitions and
basic facts of the Theory of Lie superalgebras we refer the reader to \cite{Sch79}.

In Section 1 of this paper, we shortly recall the definition of double extension and give some examples of 2-step double extensions. Although double extensions provide a useful description of quadratic Lie algebras, it is very little to know about them, for instance, the class of 2-step double extensions still remains a lot to be known and its complete classification is difficult to get. A quadratic Lie algebra is called a 2-step double extension if it is a one-dimensional double extension of a 1-step double extension but not a 1-step double extension. The notion of 1-step double extension is for quadratic Lie algebras obtained from one-dimensional double extensions of Abelian Lie algebras and such algebras have just been classified up to isomorphism and isometric isomorphism in \cite{DPU12}. We also introduce the notion of $T^*$-extensions and list all $T^*$-extensions of three-dimensional solvable Lie algebras that provide a complete classification of solvable quadratic Lie algebras of dimension 6. All of them are 1-step double extensions. Section 2 is devoted to quadratic Lie superalgebras where we give an expansion of the double extension and $T^*$-extension notions for the superalgebra case. We also give an exhausted classification of solvable quadratic Lie superalgebras up to dimension 6. In the last section, we provide a way for constructing an odd quadratic Lie superalgebra and show an isometrically isomorphic classification of solvable odd quadratic Lie superalgebras up to dimension 6.
\section{Quadratic Lie algebras}
\begin{defn}Let $\g$ be a Lie algebra. A symmetric bilinear form $B:\ \g\times\g\rightarrow \CC$ is called:
\begin{enumerate}
	\item[(i)] {\em non-degenerate} if $B(X,Y) = 0$ for all $Y\in\g$ implies $X=0$,
	\item[(ii)] {\em invariant} if $B([X,Y],Z) = B(X,[Y,Z])$ for all $X,\ Y,\ Z \in\g$.
\end{enumerate}

A Lie algebra $\g$ is called {\em quadratic} if there exists a bilinear form $B$ on $\g$ such that $B$ is symmetric, non-degenerate and invariant.
\end{defn}

Let $(\g, B)$ be a quadratic Lie algebra. Since $B$ is non-degenerate and invariant, we have some simple properties of $\g$ as follows:
\begin{prop} \label{prop1.2}\hfill
\begin{enumerate}
	\item If $I$ is an ideal of $\g$ then $I^\bot$ is also an ideal of $\g$. Moreover, if $I$ is non-degenerate then so is $I^\bot$ and $\g = I\oplus I^\bot$. Conveniently, in this case we use the notation $\g = I\oplusp I^\bot$.
	\item $\Zs(\g) = [\g,\g]^\bot$ where $\Zs(\g)$ is the center of $\g$. And then $\dim(\Zs(\g)) + \dim([\g,\g]) = \dim(\g).$
\end{enumerate}
\end{prop}

A quadratic Lie algebra $\g$ is called {\em indecomposable} if $\g = \g_1
  \oplusp \g_2$, with $\g_1$ and $\g_2$ ideals of $\g$, implies $\g_1$
  or $\g_2 = \{0\}$. Otherwise, we call $\g$ {\em decomposable}.
	
We say that two quadratic Lie algebras $(\g,B)$ and $(\g',B')$ are isometrically isomorphic (or
  i-isomorphic, for short) if there exists a Lie algebra isomorphism $A$
  from $\g$ onto $\g'$ satisfying $B'(A(X), A(Y)) = B(X,Y)$ for all $X,\ Y \in \g$.
  In this case, $A$ is called an i-isomorphism. 
\subsection{Double extensions}
\begin{defn}\label{defn1.2}
Let $(\g,B)$ be a quadratic Lie algebra and $D$ a derivation of $\g$. We say $D$ a {\em skew-symmetric} derivation of $\g$ if it satisfies $B(D(X),Y) = -B(X,D(Y))$ for all $X,\ Y\in\g$.
\end{defn}

Denote by $\Der_a(\g,B)$ the vector space of skew-symmetric derivations of $(\g,B)$ then $\Der_a(\g,B)$ is a subalgebra of $\Der(\g)$, the Lie algebra of derivations of $\g$. The notion of double extension is defined as follows (see\cite{MR85}).
\begin{defn}\label{defn1.3}
Let $\g$ be a Lie algebra, $\g^*$ its dual space and $(\hk,B)$ a quadratic Lie algebra. Let $\psi:\g\rightarrow\Der_a(\hk,B)$ be a Lie algebra endomorphism. Denote by $\phi:\hk\times\hk\rightarrow \g^*$ the linear mapping defined by $\phi(X,Y)Z = B(\psi(Z)(X),Y)$ for all $X,\ Y\in\hk,\ Z\in\g$. Consider the vector space $\overline{\hk}=\g\oplus\hk\oplus\g^*$ and define a product on $\overline{\hk}$ by:
\begin{eqnarray*}
  [X+F+f,Y+G+g]_{\hk} = &[X,Y]_\g + [F,G]_{\hk} + \ad^*(X)(g)-\ad^*(Y)(f)\\
  & +\psi(X)(G)-\psi(Y)(F)+\phi(F,G) 
\end{eqnarray*}
for all $X,Y\in\g$, $f,g\in\g^*$ and $F,G\in\hk$. Then $\overline{\hk}$ becomes a quadratic Lie algebra with the bilinear form $\overline{B}$ given by $\overline{B}(X+F+f,Y+G+g) = f(Y)+g(X)+B(F,G)
$ 
for all $X,Y\in\g$, $f,g\in\g^*$ and $F,G\in\hk$. The Lie algebra $(\overline{\hk},\overline{B})$ is called the {\em double extension of $(\hk,B)$ by $\g$ by means of $\psi$}.
\end{defn}

Note that when $\hk=\{0\}$ then this definition is reduced to the notion of the semidirect product of $\g$ and $\g^*$ by the coadjoint representation.

\begin{prop}\cite{MR85} \label{2.3}\hfill

Let $(\g, B)$ be an indecomposable quadratic Lie algebra such that it is not simple nor one-dimensional. Then $\g$ is a double extension of a quadratic Lie algebra by a simple or one-dimensional algebra.
\end{prop}

Sometimes, we use a particular case of the notion of double extension: a double extension by a skew-symmetric derivation. It is explicitly defined as follows.

\begin{defn} 
Let $(\g,B)$ be a quadratic Lie algebra and $C\in \Der_a(\g)$. On the vector space $\bar{\g}=\g\oplus \CC e\oplus \CC f$ we define the product 
 $[X,Y]_{\bar{\g}} = [X,Y]_{\g} +B(C(X),Y)f$, $[e,X]=C(X)$ and $[f,\bar{\g}]=0$ for all $X,\ Y\in\g$. Then $\bar{\g}$ is a quadratic Lie algebra with an invariant bilinear form $\bar{B}$ defined by:
$$\bar{B}(e,e)=\bar{B}(f,f)=\bar{B}(e,\g)=\bar{B}(f,\g)=0, \ \bar{B}(X,Y)=B(X,Y),\ \bar{B}(e,f)=1
$$
for all $X,\ Y\in\g$. In this case, we call $\bar{\g}$ the {\em double extension of $\g$ by $C$} or a {\em one-dimensional double extension}, for short.
\end{defn}

A remarkable result is that one-dimensional double extensions are sufficient for studying solvable quadratic Lie algebras (see \cite{Kac85} or \cite{FS87}).
\begin{ex}\label{ex1.8} Let $\g_4$ be the diamond Lie algebra spanned by $\{X,P,Q,Z\}$ where the Lie bracket is defined by: $[X,P]=P$, $[X,Q]=-Q$ and $[P,Q]=Z$ then $\g$ is quadratic with invariant bilinear form $B$ given by $B(X,Z)=B(P,Q)=1$, the other are vanish.   Assume that $D$ is a skew-symmetric derivation of $\g_4$. By a straightforward computation, the matrix of $D$ in the given basis is:
$$ D = \begin{pmatrix} 0 & 0 & 0 & 0 \\ y & x & 0 & 0 \\ z &  0 & -x & 0\\ 0 & -z & -y & 0\end{pmatrix}$$ 
where $x,\ y,\ z\in\CC$. Let $\bar{\g}_4=\g_4\oplus \CC e\oplus \CC f$ be the double extension of $\g_4$ by $D$. Then the Lie bracket is defined on $\bar{\g}_4$ as follows $[e,X]=yP+zQ$, $[e,P] = xP-zZ$, $[e,Q]=-xQ-yZ$, $[X,P]=P+zf$, $[X,Q]=-Q+yf$ and $[P,Q]=Z+xf$.
\end{ex}

In the above example, the quadratic Lie $\bar{\g}_4$ is decomposable since the elements $u=-e+xX-yP+zQ$ and $f$ are central and $B(u,f)=-1$. Note that the skew-symmetric derivation $D$ above is inner. Really, in \cite{FS96}, the authors proved a general result that any double extension by an inner derivation is decomposable. Here we give a short proof as follows.

\begin{prop}\label{prop1.9}
Let $(\g,B)$ be a quadratic Lie algebra and $C=\ad_{\g}(X_0)$ be an inner derivation of $\g$. Then the double extension $\gb$ of $\g$ by $C$ is decomposable.
\end{prop}
\begin{proof}
If $\g=\{0\}$ then the result is obvious. Assume $\g\neq\{0\}$. It is straightforward to prove that the element $e-X_0$ is central. Moreover, $B(e-X_0,f)=1$ and $f$ is central so that $\gb$ is decomposable by Proposition \ref{prop1.2}.
\end{proof}
\begin{rem}\label{rem2.1}
As we known, the diamond Lie algebra is a one-dimensional double extension of an Abelian Lie algebra. We call this type of double extensions a {\em 1-step double extension}. The above example shows that a double extension of a 1-step double extension may be a 1-step double extension. Therefore, we define a {\em 2-step double extension} if it is a one-dimensional double extension of a 1-step double extension but not a 1-step double extension. Since a fact is that the 1-step double extensions are exhaustively classified in \cite{DPU12}, it remains k-step double extensions with $k>1$. Note that every solvable quadratic Lie algebra up to dimension 6 is a 1-step double extension \cite{DLP12}.
\end{rem}

Let $\g_5$ be a non-Abelian nilpotent quadratic Lie algebra of dimension 5 spanned by $\{X_1,X_2,T,Z_1,Z_2\}$ with the Lie bracket defined by $[X_1,X_2]=T$, $[X_1,T]=-Z_2$, $[X_2,T]=Z_1$ and the bilinear form $B$ given by $B(X_i,Z_i)=B(T,T)=1$, $i=1,2$, zero otherwise. This algebra is a 1-step double extension. 
%
We can check that $D$ is a skew-symmetric derivation of $\g_5$ if and only if the matrix of $D$ in the given basis is:
$$D = \begin{pmatrix} -x & -z & 0 & 0 & 0 \\ -y & x & 0 & 0 & 0 \\ -b & -c & 0 & 0 & 0 \\ 0 & -t & b & x & y \\ t & 0 & c & z & -x \end{pmatrix}$$
where $x,\ y,\ z,\ t,\ b,\ c\in\CC$. Combined with Lemma 5.1 in \cite{Med85}: two double extensions by  skew-symmetric derivations $D$ and $D'$ respectively are i-isomorphic if $D$ and $D'$ are different by adding an inner derivation, we consider only skew-symmetric derivations having matrix:
$$ D = \begin{pmatrix} -x & -z & 0 & 0 & 0 \\ -y & x & 0 & 0 & 0 \\ 0 & 0 & 0 & 0 & 0 \\ 0 & 0 & 0 & x & y \\ 0 & 0 & 0 & z & -x \end{pmatrix}.$$
Let $\bar{\g}_5$ be the double extension of $\g_5$ by $D$ then we have the following assertion.

\begin{prop} 
If $x=y=0$ or $x=z=0$ then $\bar{\g}_5$ is a 1-step double extension. Else, $\bar{\g}_5$ is a 2-step double extension and indecomposable.
\end{prop}
\begin{proof}
If $x=y=z=0$ then $\bar{\g}_5$ is decomposable so it is obviously a 1-step double extension. We assume that $x=z=0$ and $y\neq 0$. In this case, the Lie bracket on $\bar{\g}_5$ is defined by:
$$
[e_,X_1]=-yX_2,\ [e_,Z_2]=yZ_1, \ [X_1,X_2]=T,\ [X_1,T]=-Z_2 \ \text{and}\ [X_2,T]=Z_1,
$$
and $[X_1,Z_2] = -yf$. Then one has $\bar{\g}_5=\qk\oplus\left( \CC X_1\oplus\CC Z_1\right)$ a 1-step double extension of $\qk$ spanned by $\{e,X_2,T,f,Z_2\}$ by the skew-symmetric $C:\qk\rightarrow\qk$, $C(e)=yX_2$, $C(X_2)=T$, $C(T)=-Z_2$ and $C(Z_2)=-yf$. Note that the case of $x=y=0$ and $z\neq 0$ is completely similar to the case of $x=z=0$ and $y\neq 0$.

If $y=z=0$ and $x\neq 0$ then 
one has $\left[\left[\bar{\g}_5,\bar{\g}_5\right],\left[\bar{\g}_5,\bar{\g}_5\right]\right]={\spa}\{T,Z_1,Z_2,f\}$ so $\bar{\g}_5$ is not a 1-step double extension. Remark that $\bar{\g}_5$ is also indecomposable because if there is the contrary then $\bar{\g}_5$ must be a 1-step double extension since all solvable quadratic Lie algebra up to dimension 6 are 1-step double extensions.

The remaining cases satisfy the condition $\dim\left([[\bar{\g}_5,\bar{\g}_5],[\bar{\g}_5,\bar{\g}_5]]\right)> 1$. Hence, they are all indecomposable 2-step double extensions and the result follows.
\end{proof}

Let $(\g_{2n+2},B)$ be a quadratic Lie algebra of dimension $2n+2$, $n\geq 1$, spanned by $\{X_0,...,X_n,Y_0,...,Y_n\}$ where the Lie bracket is defined by $[Y_0,X_i]=X_i$, $[Y_0,Y_i]=-Y_i$ and $[X_i,Y_i]=X_0$, $1\leq i\leq n$, the other are trivial, and the invariant bilinear form $B$ is given by $B(X_i,Y_i)=1$, $0\leq i\leq n$, zero otherwise. If $n=1$, the Lie algebra $\g_{2n+2}$ is reduced to the diamond Lie algebra. Note that $\g_{2n+2}$ is a 1-step double extension.

Let $D$ be a skew-symmetric derivation of $\g_{2n+2}$. By $\Zs(\g_{2n+2})$ and $[\g_{2n+2},\g_{2n+2}]$ stable by $D$, we can assume $D(X_0)=aX_0$, $ D(X_i)=\sum_{j=1}^n a_{ij}X_j + \sum_{j=1}^n b_{ij}Y_j +c_i X_0$ and $D(Y_i)=\sum_{j=1}^n a'_{ij}X_j + \sum_{j=1}^n b'_{ij}Y_j +c'_i X_0
$
 where the coefficients are complex numbers. Moreover, since $B(D(Y_0),X_0) = -B(Y_0,D(X_0))=a$ and $B(D(Y_0),Y_0)=0$, we can also write:
$$
D(Y_0)=-aY_0+\sum_{i=1}^n \alpha_i X_i+ \sum_{i=1}^n \beta_i Y_i.
$$

It is easy to prove that $a=0$, $b_{ij}=0$, $c_i= - \beta_i$, $a'_{ij}=0$, $c'_i=-\alpha_i$ and $a_{ij}=-b'_{ji}$. 
Therefore, we rewrite $D$:
$$
D(Y_0)=\sum_{i=1}^n \alpha_i X_i+ \sum_{i=1}^n \beta_i Y_i,\ D(X_i)=\sum_{j=1}^n a_{ij}X_j -\beta_i X_0,\
D(Y_i)=-\sum_{j=1}^n a_{ji}Y_j -\alpha_i X_0
$$
and $D(X_0)=0$.

On the contrary, we can verify that if $D$ is defined as above then $D$ is a skew-symmetric derivation of $\g_{2n+2}$.

Now we assume $n\geq 2$ and let $D$ be a particular derivation defined by $D(X_1)=X_1$, $D(Y_1)=-Y_1$ and the other are vanish. It is easy to see that $D$ can not be an inner derivation and then the double extension $\gb$ of $\g_{2n+2}$ by $D$ has the Lie bracket: $[e,X_1]=X_1$, $[e,Y_1]=-Y_1$, $[Y_0,X_i] = X_i$, $[Y_0,Y_i] = -Y_i$, $[X_1,Y_1]=X_0 + f$ and $[X_i,Y_i] = X_0$, $2\leq i\leq n$. It implies that $\left[[\gb,\gb],[\gb,\gb] \right]={\spa}\{X_0,X_0+f \}$ so that $\gb$ is a 2-step double extension.

\begin{rem} It is obvious that $\g_1\oplusp\g_2$ with $\g_1$ and $\g_2$ non-Abelian 1-step double extensions is a 2-step double extension. However, this case is rather trivial since it is decomposable. For instance, the Lie algebra $\gb$ defined as above is decomposable by $\gb = \g_4\oplusp \g_{2n}$. More explicitly, $\gb={\spa}\{e,X_1,Y_1,X_0+f\}\ \oplusp\ {\spa}\{e-Y_0,X_i,Y_i,X_0\},\ \ 2\leq i\leq n.$
\end{rem}

\subsection{$T^*$-extensions}
\begin{defn}\label{defn1.4}\cite{Bor97}\hfill

Let $\g$ be a Lie algebra and $\theta:\g\times\g\rightarrow\g^*$ a 2-cocycle of $\g$, that is a skew-symmetric bilinear map satisfying:
$$
\theta(X,Y)\circ{\ad}(Z)+\theta([X,Y],Z)+\ cycle(X,Y,Z)\ =0.
$$
 for all $X,Y,Z\in\g$. Define on the vector space $T^*_{\theta}(\g):=\g\oplus\g^*$ the product:
$$ [X+f,Y+g]=[X,Y]+{\ad}^*(X)(g)-{\ad}^*(Y)(f) +\theta(X,Y)
$$
for all $X,Y\in\g$, $f,g\in\g^*$ then $T^*_{\theta}(\g)$ becomes a Lie algebra and it is called the {\em $T^*$-extension of $\g$ by means of $\theta$}. In addition, if $\theta$ satisfies the {\em cyclic} condition, i.e. $\theta(X,Y)Z=\theta(Y,Z)X$ for all $X,Y,Z\in\g$ then $T^*_{\theta}(\g)$ is quadratic with the bilinear form:
$$
B(X+f,Y+g)=f(Y)+g(X),\ \forall X,Y\in\g,\ f,g\in\g^*.
$$
\end{defn}
%
\begin{prop}\label{2.1.13}\cite{Bor97}\hfill

Let $(\g, B)$ be an even-dimensional quadratic Lie algebra over $\CC$. If $\g$ is solvable then $\g$ is i-isomorphic to a $T^*$-extension $T_{\theta}^*(\hk)$ of $\hk$ where $\hk$ is the quotient algebra of $\g$ by a totally isotropic ideal.
\end{prop}

As in Proposition \ref{prop1.9}, a double extension by an inner derivation is decomposable. We give here a similar situation for a $T^*$-extension decomposable as follows.
\begin{prop}\label{prop1.16}
Assume that $\g$ is a Lie algebra with $\Zs(\g)\neq\{0\}$ and $\theta$ is a cyclic 2-cocycle of $\g$. If there are a nonzero $X\in\Zs(\g)$ and a nonzero $a\in\CC$ such that
$$\theta(X,Y)=aX^*\circ{\ad}(Y),\ \forall\ Y\in\g,
$$
where $X^*$ denotes the dual form of $X$, then $T^*_{\theta}(\g)$ is decomposable.
\end{prop}
\begin{proof}
We can show that if there are such $X$ and $a$ then $X-aX^*$ is central. 
Moreover, $B(X-aX^*,X-aX^*)=-2a\neq 0$ then $T^*_{\theta}(\g)$ is decomposable.
\end{proof}

\begin{ex}\label{ex1.6}
Denote by $\hk_3$ the Lie algebra spanned by $\{X,Y,Z\}$ such that $[X,Y]=Z$. Assume $\{X^*,Y^*,Z^*\}$ the dual basis of $\{X,Y,Z\}$ and let $\theta$ be a non-trivial cyclic 2-cocycle of $\hk_3$. It is easy to compute that $\theta(X,Y)=\lambda Z^*$, $\theta(Y,Z)=\lambda X^*$ and $\theta(Z,X)=\lambda Y^*$ where $\lambda$ is nonzero in $\CC$. Then $T^*_{\theta}(\hk_3)$ is decomposable since $\theta(Z,T)=\lambda Z^*\circ{\ad}(T)$ for all $T\in\hk_3$.
\end{ex}
\begin{rem} \label{rem1.7}
Let $\g$ be a Lie algebra, $\theta$ a cyclic 2-cocycle of $\g$ and a nonzero $\lambda\in\CC$. Then two $T^*$-extensions $T^*_{\theta}(\g)$ and $T^*_{\lambda\theta}(\g)$ are isomorphic by the isomorphism $A:T^*_{\theta}(\g)\rightarrow T^*_{\lambda\theta}(\g)$, $A(X+f)=X+\lambda f$.
\end{rem}
\subsection{$T^*$-extensions of three-dimensional solvable Lie algebras}
\hspace{0.8cm}As we known, if $\g$ is a solvable Lie algebra of dimension 3 then it is isomorphic to each of the following Lie algebras:

\begin{enumerate}
\item $\g_{3,0}$: Abelian,
	\item $\g_{3,1}$: $[X,Y]=Z$,
	\item $\g_{3,2}$: $[X,Y]=Y$ and $[X,Z]=Y+Z$,
	\item $\g_{3,3}$: $[X,Y]=Y$ and $[X,Z]=\mu Z$ with $|\mu|\leq 1$.
\end{enumerate}
\begin{prop}\label{prop1.19} Let $\g$ be a solvable Lie algebra of dimension 3 and $\theta$ a cyclic 2-cocycle of $\g$. If the $T^*$-extension of $\g$ by means of $\theta$ is indecomposable then it is i-isomorphic to each of quadratic Lie algebras as follows:

\begin{enumerate}
	\item $\g_{6,1}$: $[X,Y]=Z$, $[X,Z^*]=-Y^*$ and $[Y,Z^*]=X^*$,
	\item $\g_{6,2}$: $[X,Y]=Y$, $[X,Z]=Y+Z$, $[X,Y^*]=-Y^*-Z^*$, $[X,Z^*]=-Z^*$ and $[Y,Y^*]=[Z,Y^*]=[Z,Z^*]=X^*$,
	\item $\g_{6,3}$: $[X,Y]=Y$, $[X,Z]=\mu Z$, $[X,Y^*]=-Y^*$, $[X,Z^*]=-\mu Z^*$, $[Y,Y^*]=X^*$ and $[Z,Z^*]=\mu X^*$ where $|\mu|\leq 1$, $\mu \neq -1$.
\end{enumerate}
\end{prop}

\begin{rem}
The result given in Proposition \ref{prop1.19} is a complete classification for the solvable case of dimension 6 (see \cite{DLP12}).
\end{rem}
\section{Quadratic Lie superalgebras}
\begin{defn}Let $\g = \g_\ze\oplus\g_\un$ be a Lie superalgebra. We assume that there is a non-degenerate invariant supersymmetric bilinear form $B$ on $\g$, i.e. $B$ satisfies:
\begin{enumerate}
	\item $B(Y,X) = (-1)^{xy}B(X,Y)$ (supersymmetric) for all $X\in\g_x,\ Y\in\g_y$,
	\item $B([X,Y],Z) = B(X,[Y,Z])$ (invariant) for all $X,Y,Z \in\g.$
\end{enumerate}

If $B$ is even, that is $B(\g_\ze,\g_\un) =0$, then we call $\g$ a \emph{quadratic Lie superalgebra}. In this case, $B_0=B|_{\g_\ze\times\g_\ze}$ and $B_1=B|_{\g_\un\times\g_\un}$ are non-degenerate, consequently $(\g_\ze,B_0)$ is a quadratic Lie algebra and $(\g_\un,B_1)$ is a symplectic vector space. The invariancy of $B$ implies that $B_1({\ad}_X(Y),Z) = -B_1(Y,{\ad}_X(Z))$ for all $X\in\g_\ze,\ Y,Z\in\g_\un$. Therefore, the adjoint representation is a homomorphism from the Lie algebra $\g_\ze$ into the Lie subalgebra $\spk(\g_\un,B_1)$ of $\End(\g_\un,\g_\un)$.

If $B$ is odd, i.e. $B(\g_\alpha,\g_\alpha)=0$ for all $\alpha\in\ZZ_2$, then we call $\g$ an \emph{odd quadratic Lie superalgebra}. In this case, $\g_\ze$ and $\g_\un$ are totally isotropic subspaces of $\g$ and $\dim(\g_\ze) = \dim(\g_\un)$.
\end{defn}

Proposition \ref{prop1.2} is still right here so that we have similarly sequential properties for quadratic Lie superalgebras and odd quadratic Lie superalgebras.


For the notion of 1-step double extension in the quadratic Lie superalgebras case, we refer the reader to \cite{DU12} (the definition of double extension in the general case is given in \cite{BB99}):

%
Clearly, the definition of the semidirect product of a Lie algebra $\g$ and its dual space $\g^*$ by the coadjoint representation can be comprehended that the Lie algebra $\g$ is "glued" by $\g^*$ while the notion of double extension is more general by such "gluing" but combined with a quadratic Lie algebra $(\hk,B)$ to get a new quadratic Lie algebra. From this point of view, what happens when we glue the dual space $\g^*$ to $\g$ combined with a {\em symplectic} vector space $(\hk,B)$ to obtain a quadratic Lie superalgebra? Obviously, in this case, the even part is $\g\oplus\g^*$ and the odd is $\hk$.

The following lemma is straightforward.
\begin{lem}\label{lem1.8}
Let $\g$ be a Lie algebra and $(\hk,B_{\hk})$ a symplectic vector space with symplectic form $B_{\hk}$. Let $\psi:\g\rightarrow\End(\hk)$ be a Lie algebra endomorphism satisfying:
$$
B_{\hk}(\psi(X)(Y),Z)=-B_{\hk}(Y,\psi(X)(Z)),\ \forall\ X\in\g,\ Y,Z\in\hk.
$$
Denote by $\phi:\hk\times\hk\rightarrow\g^*$ the bilinear map defined by $\phi(X,Y)Z=B_{\hk}(\psi(Z)(X),Y)$ for all $X,\ Y\in\hk$, $Z\in\g$ then $\phi$ is symmetric. \end{lem}

\begin{prop}\label{prop2.4}
Keep notions as in the above lemma and define on the vector space $\gb:=\g\oplus\g^*\oplus\hk$ the bracket:
\begin{eqnarray*}
  [X+f+F,Y+g+G]_{\gb} = &[X,Y]_\g + {\ad}^*(X)(g)-{\ad}^*(Y)(f)\\
  & +\psi(X)(G)-\psi(Y)(F)+\phi(F,G) 
\end{eqnarray*}
for all $X,Y\in\g$, $f,g\in\g^*$ and $F,G\in\hk$. Then $\gb$ becomes a quadratic Lie superalgebra with $\gb_\ze=\g\oplus\g^*$ and $\gb_\un=\hk$ with the bilinear form $\overline{B}$ defined by:
$$
\overline{B}(X+f+F,Y+g+G)=f(Y)+g(X)+B_{\hk}(F,G)
$$
for all $X,Y\in\g$, $f,g\in\g^*$ and $F,G\in\hk$.
\end{prop}

\begin{proof}
The proof is a straightforward but lengthy computation so we omit it.
\end{proof}

Now we combine the above proposition with Definition \ref{defn1.4} and Proposition \ref{prop1.9} to get a more general result as follows.
\begin{prop}
Let $\g$ be a Lie algebra and $\theta:\g\times\g\rightarrow\g^*$ a 2-cocycle of $\g$. Assume $(\hk,B_{\hk})$ a symplectic vector space with symplectic form $B_{\hk}$. Let $\psi:\g\rightarrow\End(\hk)$ be a Lie algebra endomorphism satisfying $B_{\hk}(\psi(X)(Y),Z)=-B_{\hk}(Y,\psi(X)(Z))$ for all $X\in\g$, $Y,Z\in\hk$. Denote by $\phi:\hk\times\hk\rightarrow\g^*$ the bilinear map defined by $\phi(X,Y)Z=B_{\hk}(\psi(Z)(X),Y)$ for all  $X,\ Y\in\hk$, $Z\in\g$ and define on the vector space $\gb:=\g\oplus\g^*\oplus\hk$ the bracket:
\begin{eqnarray*}
  [X+f+F,Y+g+G]_{\gb} = &[X,Y]_\g + {\ad}^*(X)(g)-{\ad}^*(Y)(f)+\theta(X,Y)\\
  & +\psi(X)(G)-\psi(Y)(F)+\phi(F,G) 
\end{eqnarray*}
for all $X,Y\in\g$, $f,g\in\g^*$ and $F,G\in\hk$. Then $\gb$ becomes a Lie superalgebra with $\gb_\ze=\g\oplus\g^*$ and $\gb_\un=\hk$. Moreover, if $\theta$ is cyclic then $\gb$ is a quadratic Lie superalgebra with the bilinear form $\overline{B}(X+f+F,Y+g+G)=f(Y)+g(X)+B_{\hk}(F,G)$ for all $X,Y\in\g$, $f,g\in\g^*$ and $F,G\in\hk$.
\end{prop}
\subsection{Quadratic Lie superalgebras of dimension 4}
\hspace{0.8cm}Let $\g$ be a quadratic Lie superalgebra. If $\dim(\g_\ze) \leq 1$ then $\g$ is Abelian \cite{DU12}. Therefore, in the non-Abelian four-dimensional case, we only consider $\dim(\g_\un) =0$ or $\dim(\g_\ze) = \dim(\g_\un) = 2$. 

If $\dim(\g_\un) =0$ then $\g$ is isomorphic to $\ok(3)\oplus\CC$ or the diamond Lie algebra $\g_4$. 

If $\dim(\g_\ze) = \dim(\g_\un) = 2$ then there is  a classification up to i-isomorphism given in \cite{DU12} as follows:

\begin{enumerate}
	\item $\g_{4,1}^s = \g_\ze\oplus\g_\un$, where $\g_\ze ={\spa}\{X_{\zero}, Y_{\zero}\}$ and $\g_\un={\spa}\{X_\un, Y_\un\}$ such that the non-zero bilinear form $B(X_{\zero}, Y_{\zero})=B(X_{\un}, Y_{\un})=1$ and the non-trivial Lie super-bracket $[Y_{\un},Y_{\un}]=-2X_{\zero}$, $[Y_{\zero},Y_{\un}]=-2X_{\un}$.
\item $\g_{4,2}^s = \g_\ze\oplus\g_\un$, where $\g_\ze ={\spa}\{X_{\zero}, Y_{\zero}\}$ and $\g_\un={\spa}\{X_\un, Y_\un\}$ such that the non-zero bilinear form $B(X_{\zero}, Y_{\zero})=B(X_{\un}, Y_{\un})=1$ and the non-trivial Lie super-bracket $[X_{\un},Y_{\un}]=X_{\zero}$, $[Y_{\zero},X_{\un}]=X_{\un}$, $[Y_{\zero},Y_{\un}]=-Y_{\un}$.
\end{enumerate}

\begin{rem}
We can obtain the above result by the method of double extension as follows: $\g$ can be seen as a double extension of the two-dimensional symplectic vector space $\qk$ by a map $C\in\spk(\qk)$. Since there are only two cases: $C$ nilpotent or $C$ semi-simple, the result follows.
\end{rem}

\subsection{Quadratic Lie superalgebras of dimension 5}
\hspace{0.8cm}First we recall a result in \cite{BB99} that is useful for our classification as follows.
\begin{prop}\label{prop2.7}
Let $\g$ be an indecomposable quadratic Lie superalgebra such that $\dim(\g_\un)=2$ and $\Zs(\g_\ze)\neq\{0\}$, where $\Zs(\g_\ze)=\{X\in\g_\ze\ |\ [X,Y]=0$ for all $Y\in\g_\ze\}$. Then $\Zs(\g)\cap\g_\ze\neq\{0\}$. Moreover, if $\dim([\g_\un,\g_\un])\geq 2$ then $\Zs(\g_\ze)\subset\Zs(\g)$.
\end{prop}

Let $\g$ be an indecomposable quadratic Lie superalgebras of dimension 5 and assume $\g_\un \neq \{0\}$ then we have $\dim(\g_\ze) = 3$ and $\dim(\g_\un) = 2$. 

If $\g_\ze$ is solvable then $\g_\ze$ must be Abelian \cite{PU07}. In this case, we need the following lemma.

\begin{lem}\label{lem2.8} 
Let $\g$ be an indecomposable quadratic Lie superalgebra such that $\g_\ze\neq\{0\}$. Assume that $\g_\ze$ is Abelian and $\dim(\g_\un) = 2$. Then one has $\dim(\g_\ze) = 2$.
\end{lem}
		
By the above lemma, if $\g$ is a five-dimensional indecomposable quadratic Lie superalgebra with $\g_\ze$ is solvable then $\g$ must be i-isomorphic to the five-dimensional nilpotent quadratic Lie algebra $\g_5$.

If $\g_\ze$ is not solvable then $\g_\ze\simeq \ok(3)$. Choose a basis $\{X_1,X_2,X_3\}$ of $\g_\ze$ such that $B(X_i,X_j)=\delta_{ij}$, $[X_1,X_2] = X_3$, $[X_2,X_3] = X_1$ and $[X_3,X_1] = X_2$. Recall that the adjoint representation $\ad$ is a homomorphism from $\g_\ze$ onto $\spk(\qk_\un)$. If it is not an isomorphism, i.e we can assume that ${\ad}(X_1) = x{\ad}(X_2)+y{\ad}(X_3)$ then
$${\ad}(X_3) = [{\ad}(X_1),{\ad}(X_2)] = -y{\ad}(X_1)\ \text{and}\ {\ad}(X_2) = [{\ad}(X_3),{\ad}(X_1)] = -x{\ad}(X_1).
$$
That implies ${\ad}(X_2)=[{\ad}(X_1),{\ad}(X_3)]=0$. Similarly, ${\ad}(X_3) = 0$ and so ${\ad}(X_1) = 0$. It implies that $[\g_\ze,\g_\un] = \{0\}$, this is a contradiction since $\g$ is indecomposable. Therefore, the map $\ad$ must be an isomorphism from $\g_\ze$ onto $\spk(\qk_\un)$ and $\g\simeq \ospk(1,2)$.

\subsection{Solvable quadratic Lie superalgebras of dimension 6}
\hspace{0.8cm} If $\dim(\g_\un) =0$ then the classification in the solvable case is given in Proposition \ref{prop1.19}. We only consider two non trivial cases:  $\dim(\g_\un) = 2$ and $\dim(\g_\un) = 4$ as follows.
\subsubsection{$\dim(\g_\un) = 2$.} 
We assume that $\g$ is indecomposable. By Lemma \ref{lem2.8}, $\g_\ze$ is non-Abelian. If $\g_\ze$ is not solvable then $\g_\ze = \sk\oplus \rk$ with $\sk$ semi-simple, $\rk$ the radical of $\g$ and $\sk\simeq \slk(2)$. Since $B([\sk,\sk^\bot],\sk) = B([\sk,\sk],\sk^\bot) = B(\sk,\sk^\bot) = 0$, then ${\ad}:\sk\rightarrow \End(\sk^\bot,\sk^\bot)$ is a one-dimensional representation of $\sk$. Hence, ${\ad}(\sk)|_{\sk^\bot} = 0$ and $[\sk^\bot,\sk]=\{0\}$. It implies that $\sk$ and $\sk^\bot$ are non-degenerate ideals of $\g_\ze$. Note that $\sk^\bot=\Zs(\g_\ze)$. By Proposition \ref{prop2.7}, $\Zs(\g)\cap\g_\ze\neq \{0\}$ so $\sk^\bot\subset \Zs(\g)$ and then $\g$ is decomposable. This is a contradiction. Therefore, $\g_\ze$ is solvable and i-isomorphic to the diamond Lie algebra $\g_4$ give in Example \ref{ex1.8}. Particularly, $\g_\ze={\spa}\{X,P,Q,Z\}$ such that $B(X,Z) = B(P,Q)=1$, $B(X,Q) = B(P,Z)=0$, the subspaces spanned by $\{X,P\}$ and $\{Q,Z\}$ are totally isotropic. The Lie bracket on $\g_\ze$ is defined by $[X,P] = P$, $[X,Q] = -Q$ and $[P,Q] = Z$. It is obvious that $(\Zs(\g)\cap\g_\ze)\subset\Zs(\g_\ze)$. By Proposition \ref{prop2.7}, one has $(\Zs(\g)\cap\g_\ze)=\CC Z$. 

If ${\ad}(X)|_{\g_\un} = 0$ then ${\ad}(P)|_{\g_\un} = {\ad}([X,P])|_{\g_\un} = [{\ad}(X),{\ad}(P)]|_{\g_\un} = 0$. Similarly, ${\ad}(Q)|_{\g_\un} = 0$ and then $[\g_\ze,\g_\un] = 0$. It implies $[\g_\un,\g_\un] = 0$ by the invariance of $B$. This is a contradiction since $\g$ is indecomposable. Hence, ${\ad}(X)|_{\g_\un} \neq 0$.

Note that ${\ad}(X)|_{\g_\un}$, ${\ad}(P)|_{\g_\un}$ and ${\ad}(Q)|_{\g_\un} \in\spk(\g_\un,B|_{\g_\un\times\g_\un}) = \spk(2)$, where
$$\spk(2) = {\spa}\left\{ \begin{pmatrix} 1 & 0 \\ 0 & -1\end{pmatrix}, \ \begin{pmatrix} 0 & 1 \\ 0 & 0\end{pmatrix}, \ \begin{pmatrix} 0 & 0 \\ 1 & 0\end{pmatrix}\right\}.$$
Set the subspace $V = {{\spa}}\{{\ad}(X)|_{\g_\un}, {\ad}(P)|_{\g_\un}, {\ad}(Q)|_{\g_\un}\}$ of $\spk(2)$. Assume that
$${\ad}(P)|_{\g_\un} = a \begin{pmatrix} 1 & 0 \\ 0 & -1\end{pmatrix}+b \begin{pmatrix} 0 & 1 \\ 0 & 0\end{pmatrix} + c\begin{pmatrix} 0 & 0 \\ 1 & 0\end{pmatrix}$$
$$\text{and}\ {\ad}(Q)|_{\g_\un} = a' \begin{pmatrix} 1 & 0 \\ 0 & -1\end{pmatrix}+b' \begin{pmatrix} 0 & 1 \\ 0 & 0\end{pmatrix} + c'\begin{pmatrix} 0 & 0 \\ 1 & 0\end{pmatrix}.$$
Since $[{\ad}(P)|_{\g_\un},{\ad}(Q)|_{\g_\un}] = {\ad}([P,Q])|_{\g_\un} = {\ad}(Z)|_{\g_\un}=0$, one has
$\left\{\begin{matrix} ab'-a'b=0 \\ ac'-a'c= 0  \\ bc'-b'c=0\end{matrix}\right.$.
That means ${\ad}(P)|_{\g_\un}$ and ${\ad}(Q)|_{\g_\un}$ not linearly independent and then $1\leq\dim(V)\leq 2$. We consider two following cases:

\begin{enumerate}
	\item If $\dim(V)=1$ then ${\ad}(P)|_{\g_\un} = \alpha{\ad}(X)|_{\g_\un}$ and ${\ad}(Q)|_{\g_\un} = \beta{\ad}(X)|_{\g_\un}$, where $\alpha,\ \beta\in\CC$. Since ${\ad}(P) = [{\ad}(X),{\ad}(P)]$ and ${\ad}(Q) = -[{\ad}(X),{\ad}(Q)]$, one has $\alpha=\beta=0$. As a consequence, we obtain $B([\g_\ze,\g_\ze],[\g_\un,\g_\un]) = B([[\g_\ze,\g_\ze],\g_\un],\g_\un) = 0$.
It implies $[\g_\un,\g_\un]\subset \CC Z$ since $\CC Z$ is the orthogonal complementary of $[\g_\ze,\g_\ze]$ in $\g_\ze$. It is easy to see that $\g$ is the double extension of the quadratic $\ZZ_2$-graded vector space $\qk=(\CC X\oplus\CC Z)^\bot$ of dimension 4 by the map $C={\ad}(X)|_{\qk}$. Write $\qk=\qk_\ze\oplus\qk_\un$ and then $C$ is a linear map in $\ok(\qk_\ze)\oplus\spk(\qk_\un)$ defined by:
	$$C=\begin{pmatrix} 1 & 0 & 0 & 0 \\ 0 & -1 & 0 & 0  \\ 0 & 0 & \lambda & a \\ 0 & 0 & b & -\lambda \end{pmatrix}	$$
	where $C_\ze =  \begin{pmatrix} 1 & 0 \\ 0 & -1 \end{pmatrix}\in\ok(\qk_\ze)$ and $C_\un =  \begin{pmatrix} \lambda & a \\ b & -\lambda \end{pmatrix}\in\spk(\qk_\un)$. 
	By \cite{DU12}, let $\g$ and $\g'$ be two double extension of $\qk$ by $C$ and $C'$ then $\g$ and $\g'$ is i-isomorphic if and only if $C_\ze$ and $\mu C'_\ze$ are on the same $\OO(\qk_\ze)$-orbit of $\ok(\qk_\ze)$, $ C_\un$ and $\mu C'_\un$ are on the same $\Sp(\qk_\un)$-orbit of $\spk(\qk_\un)$ for some nonzero $\mu\in\CC$. Therefore, we have only two non i-isomorphic families of Lie superalgebras corresponding to matrices
	$$\begin{pmatrix} 1 & 0 & 0 & 0 \\ 0 & -1 & 0 & 0  \\ 0 & 0 & 0 & 1 \\ 0 & 0 & 0 & 0 \end{pmatrix},\ \ \begin{pmatrix} 1 & 0 & 0 & 0 \\ 0 & -1 & 0 & 0  \\ 0 & 0 & \lambda & 0 \\ 0 & 0 & 0 & -\lambda \end{pmatrix},\ \lambda \neq 0.	$$
	As a consequence, we obtain Lie superalgebras as follows:
	
	\begin{itemize}
		\item[(i)] $\g_{6,1}^s$: $\g_\ze = \g_4$, $[X,Y_\un] = X_\un$, $[Y_\un,Y_\un]=Z$,
		\item[(ii)] $\g_{6,2}^s(\lambda)$: $\g_\ze = \g_4$, $[X,X_\un] = \lambda X_\un$, $[X,Y_\un] = -\lambda Y_\un$ and $[X_\un,Y_\un]=\lambda Z$,
	 \end{itemize}
	where $\g_\un={\spa}\{X_\un,Y_\un\}$ and $B(X_\un,Y_\un) = 1$. For (ii), $\g_{6,2}^s(\lambda_1)$ and $\g_{6,2}^s(\lambda_2)$ are i-isomorphic if and only if $\lambda_1=\lambda_2$. Moreover, two isomorphic and i-isomorphic notions are equivalent.
	\item If $\dim(V)=2$. Assume that ${\ad}(X)|_{\g_\un} = x {\ad}(P)|_{\g_\un} + y{\ad}(Q)|_{\g_\un}$, where $x,y\in\CC$, then ${\ad}(P)|_{\g_\un}=[{\ad}(X),{\ad}(P)]|_{\g_\un} = 0$.
	Similarly, ${\ad}(Q)|_{\g_\un}=0$. It implies that ${\ad}(X)|_{\g_\un} = 0$. This is a contradiction. Therefore, we can assume that ${\ad}(Q)|_{\g_\un} = x {\ad}(X)|_{\g_\un} + y {\ad}(P)|_{\g_\un}$ (since ${\ad}(P)|_{\g_\un}$ and ${\ad}(Q)|_{\g_\un}$ play the same roll). One has:
	$$0 = [{\ad}(P)|_{\g_\un},{\ad}(Q)|_{\g_\un}] = x[{\ad}(P)|_{\g_\un}, {\ad}(X)|_{\g_\un}] = -x{\ad}(P)|_{\g_\un}.$$
	Hence, $x = 0$ and ${\ad}(Q)|_{\g_\un} =y {\ad}(P)|_{\g_\un}$.
	
	On the other hand, we have ${\ad}(Q)|_{\g_\un}=[{\ad}(Q)|_{\g_\un},{\ad}(X)|_{\g_\un}] = -y{\ad}(P)|_{\g_\un}$. That means $y=0$ and then ${\ad}(Q)|_{\g_\un} = 0$.
	
	We need the following lemma. Its proof is lengthy but straightforward so we omit it.
\begin{lem}\label{lem1}
Let $A,B$ be two nonzero linear operators in $\spk(V)$, where $\dim(V) = 2$, such that $[A,B] = B$ then $A$ is semi-simple and $B$ is nilpotent.
\end{lem}
%
%

Apply this lemma, we choose a basis $\{X_\un, Y_\un\}$ of $\g_\un$ such that $B(X_\un, Y_\un) = 1$ then ${\ad}(X)|_{\g_\un} = \begin{pmatrix} \frac{1}{2} & 0 \\ 0 & -\frac{1}{2}\end{pmatrix}$ and ${\ad}(P)|_{\g_\un} =\begin{pmatrix} 0 & \mu \\ 0 & 0\end{pmatrix}$ where $\mu\neq 0$.

By the invariancy of $B$, we have $[X_\un, Y_\un] = \frac{1}{2}Z$, $[Y_\un, Y_\un] = \mu Q$. Set $P':=\frac{P}{\mu}$ and $Q':=\mu Q$ then we have the following Lie superalgebra $\g_{6,3}^s$: $\g_\ze=\g_4$, ${\ad}(X)|_{\g_\un} = \begin{pmatrix} \frac{1}{2} & 0 \\ 0 & -\frac{1}{2}\end{pmatrix}$, ${\ad}(P)|_{\g_\un} =\begin{pmatrix} 0 & 1 \\ 0 & 0\end{pmatrix}$, $[X_\un, Y_\un] = \frac{1}{2}Z$ and $[Y_\un, Y_\un] = Q$.
\end{enumerate}

\subsubsection{$\dim(\g_\un) = 4$.} 
In this case, $\dim(\g_\ze) = 2$ and if $\g$ is non-Abelian then $\g$ is a double extension \cite{DU12}. We can choose a basis $\{X_\ze,Y_\ze\}$ of $\g_\ze$ and a canonical basis $\{X_1,X_2,Y_1,Y_2\}$ of $\g_\un$ such that $B(X_i,X_j) = B(Y_i,Y_j) = 0$, $B(X_i,Y_j) = \delta_{ij}$, $i,j = 1,2$ such that $X_\ze\in\Zs(\g)$ and $C:={\ad}(Y_\ze)|_{\g_\un}\in\spk(\g_\un,B|_{\g_\un\times\g_\un}) = \spk(4)$. Moreover, the isomorphic classification of $\g$ reduces to the classification of $\Sp(\g_\un)$-orbits of $\ps(\spk(\g_\un))$. In particular, if $C$ is nilpotent then:
	
	\begin{itemize}
		\item $C = \begin{pmatrix} 0 & 1 & 0 & 0 \\ 0 & 0 & 0 & 0 \\ 0 &  0 & 0 & 0\\ 0 & 0 & -1 & 0\end{pmatrix} $
that is corresponding to the partition $[2^2]$ of $4$ (see detail in \cite{CM93}).
\end{itemize}

If $C$ is not nilpotent. Apply the classification of $\Sp(4)$-orbits of $\ps(\spk(4))$ in \cite{DU12} one has following cases:
\begin{itemize}
\item $C = \begin{pmatrix} 0 & 0 & 1 & 0 \\ 0 & 1 & 0 & 0 \\ 0 &  0 & 0 & 0\\ 0 & 0 & 0 & -1\end{pmatrix},$
\item $C = \begin{pmatrix} 1 & 0 & 0 & 0 \\ 0 & \lambda & 0 & 0 \\ 0 &  0 & -1 & 0\\ 0 & 0 & 0 & -\lambda\end{pmatrix}, \ \lambda\neq 0,$
\item $C = \begin{pmatrix} 1 & 1 & 0 & 0 \\ 0 & 1 & 0 & 0 \\ 0 &  0 & -1 & 0\\ 0 & 0 & -1 & -1\end{pmatrix}$
	\end{itemize}
	in a canonical basis $\{X_1,X_2,Y_1,Y_2\}$ of $\g_\un$.

We obtain corresponding Lie superalgebras:
\begin{itemize}
	\item $\g_{6,4}^s$: $[Y_\ze,X_2] = X_1$, $[Y_\ze,Y_1] = -Y_2$ and $[X_2,Y_1] = X_\ze$.
	\item $\g_{6,5}^s$: $[Y_\ze,X_2] = X_2$, $[Y_\ze,Y_1] = X_1$, $[Y_\ze,Y_2] = -Y_2$ and $[Y_1,Y_1] = [X_2,Y_2] = X_\ze$.
	\item $\g_{6,6}^s(\lambda)$: $[Y_\ze,X_1] = X_1$, $[Y_\ze,X_2] = \lambda X_2 $, $[Y_\ze,Y_1] = - Y_1 $, $[Y_\ze,Y_2] = - \lambda Y_2$, $[X_1,Y_1] =X_\ze$ and $[X_2,Y_2] = \lambda X_\ze$. In this case, $\g_{6,6}^s(\lambda_1)$ is i-isomorphic to $\g_{6,6}^s(\lambda_2)$ if and only if there exists $\mu\in\CC$ nonzero such that $C(\lambda_1)$ is in the $\Sp(\g_\un)$-adjoint orbit through $\mu C(\lambda_2)$. That happens if and only if $\lambda_1=\pm\lambda_2$ or $\lambda_2=\pm{\lambda_1}^{-1}$.
	\item $\g_{6,7}^s$: $[Y_\ze,X_1] = X_1$, $[Y_\ze,X_2] =  X_2 + X_1$, $[Y_\ze,Y_1] = - Y_1 -Y_2$, $[Y_\ze,Y_2] = -Y_1$ and $[X_1,Y_1] = [X_2,Y_1] = [X_2,Y_2] = X_\ze$.
\end{itemize}
\section{Odd quadratic Lie superalgebras}
\begin{defn}
Let $\g=\g_\ze\oplus\g_\un$ be a Lie superalgebra. If there is a non-degenerate supersymmetric bilinear form $B$ on $\g$ such that $B$ is odd and invariant then the pair $(\g,B)$ is called an {\em odd-quadratic Lie superalgebra}.
\end{defn}

We have the following proposition.
\begin{prop}\label{prop3.2}
Let $\g$ be a Lie algebra and $\phi:\g^*\times\g^*\rightarrow\g$ a symmetric bilinear map satisfying two conditions:
\begin{enumerate}
	\item ${\ad}(X)\left( \phi(f,g)\right)+\phi\left(f,g\circ{\ad}(X) \right)+\phi\left(g,f\circ{\ad}(X) \right) = 0$,
	\item $f\circ{\ad}\left( \phi(g,h)\right) + \ \ cycle(f,g,h)\ \ = 0$
	for all $X\in\g$ and $f,g\in\g^*$.
\end{enumerate}
Then the vector space $\gb=\g\oplus\g^*$ with the bracket
$$
[X+f,Y+g]=[X,Y]+{\ad}^*(X)(g)-{\ad}^*(Y)(f)+\phi(f,g)
$$
for all $X,Y\in\g$, $f,g\in\g^*$ is a Lie superalgebra and it is called the {\em $T_s^*$-extension of $\g$ by means of $\phi$}. Moreover, if $\phi$ satisfies the {\bf cyclic} condition: $h\left(\phi(f,g) \right) = f\left(\phi(g,h) \right)$ for all $f,\ g,\ h\in\g^*$ then $\g$ is odd-quadratic with the bilinear form 
$$
\overline{B}(X+f,Y+g)=f(Y)+g(X),\ \forall X,Y\in\g,\ f,g\in\g^*.
$$
\end{prop}

\begin{ex} Let $\g$ be a two-dimensional Lie algebra spanned by $\{X,Y\}$ with $[X,Y]=Y$. Assume $\phi$ is a symmetric bilinear form satisfying the conditions as in Proposition \ref{prop3.2}. It is easy to compute that in this case $\phi$ must be zero and then the $T_s^*$-extension of $\g$ coincides with the semidirect product of $\g$ and $\g^*$ with the Lie bracket defined by $[X,Y] = Y$, $[X,Y^*]=-Y^*$ and $[Y,Y^*]=X^*$. This is regarded as a "superalgebra" type of the diamond Lie algebra.
\end{ex}
\begin{ex}
If $\g$ spanned by $\{X,Y\}$ is the two-dimensional Abelian Lie algebra. We can check easily that every $\phi$ having the above properties must be defined as follows.
$$
\phi(X^*,X^*)=\alpha X+\beta Y,\ \phi(X^*,Y^*)=\beta X+\gamma Y,\ \ \phi(Y^*,Y^*)=\gamma X+\lambda Y
$$
where $\alpha,\ \beta,\ \gamma,\ \lambda\in\CC$. In this case, the $T_s^*$-extension of $\g$ by means of $\phi$ is defined by $[X^*,X^*] = \alpha X+\beta Y$, $[X^*,Y^*]=\beta X+\gamma Y$ and $[Y^*,Y^*]=\gamma X+\lambda Y$. Really, this is a "Lie superalgebra" type of 2-nilpotent commutative algebras of dimension 4. Such algebras will be completely classified in the next.
\end{ex}
\begin{prop}
Let $\g$ be an odd-quadratic Lie superalgebra then $\g$ is i-isomorphic to a $T_s^*$-extension of $\g_\ze$.
\end{prop}
\begin{proof}
Assume $\g=\g_\ze\oplus \g_\un$ is an odd-quadratic Lie superalgebra with the bilinear form $B$. We can identify $\g_\un$ with $\g_\ze^*$ by the bilinear form $B$. Set the symmetric bilinear map $\phi:\g_\un\times\g_\un\rightarrow\g_\ze$ by $\phi(f,g):=[f,g]$ for all $f,\ g\in \g_\un$ then it is easy to check that $\phi$ satisfies the conditions in Proposition \ref{prop3.2} and $\g$ is the $T_s^*$-extension of $\g_\ze$ by means of $\phi$.
\end{proof}

\subsection{Odd quadratic Lie superalgebras of dimension 2}
\hspace{0.8cm}We recall the classification of odd quadratic Lie superalgebras of dimension 2 in \cite{ABB10} as follows. Let $\g=\g_\ze\oplus\g_\un$, where $\g_\ze=\CC X_\ze$ and $\g_\un=\CC X_\un$. Define an odd bilinear form on $\g$ by $B(X_\ze,X_\ze) = B(X_\un,X_\un)=0,\ B(X_\ze,X_\un) = 1.
$
Then $\g$ is Abelian or isomorphic to $\g^o_2(\lambda)$, where the nonzero Lie bracket on $\g^o_2(\lambda)$ is given by: $[X_\un,X_\un] = \lambda X_\un$. 

By a straightforward checking, we can see that $\g^o_2(1)$ is isomorphic to $\g^o_2(\lambda)$. Moreover, they are also i-isomorphic by the following i-isomorphism $A(X_\ze)=\lambda^{\frac{1}{3}}X_\ze$ and $A(X_\un)=\lambda^{-\frac{1}{3}}X_\un$. Really, these algebras were classified in \cite{DU10}, Example 3.18 in term of $T^*$-extension of the one-dimensional Abelian algebra.
\subsection{Odd quadratic Lie superalgebras of dimension 4}
\hspace{0.8cm}Now, let $\g=\g_\ze\oplus\g_\un$ be an odd quadratic Lie superalgebra, where $\g_\ze=\CC X_\ze\oplus\CC Y_\ze$ and $\g_\un=\CC X_\un\oplus\CC Y_\un$. If $\g_\ze$ is Abelian then $[\g_\ze,\g_\un] = \{0\}$ by the invariance of $B$. This case have just been classified in \cite{DU10}. In particular, if $\g$ is non-Abelian then it is i-isomorphic to each of algebras as follows:

\begin{itemize}
	\item $\g^o_{4,1}:\ [X_\un,X_\un] = Y_\ze,\ [X_\un,Y_\un]=X_\ze$,
	\item $\g^o_{4,2}:\ [X_\un,X_\un] = Y_\ze,\ [X_\un,Y_\un]=X_\ze+Y_\ze$ and $[Y_\un,Y_\un]=X_\ze$.
\end{itemize}
Note that $\g^o_{4,1}$ and $\g^o_{4,2}$ are not isomorphic.

If $\g_\ze$ is non-Abelian then we can assume $[X_\ze,Y_\ze]=Y_\ze$. By $B(X_\ze,[X_\ze,\g_\un])=0$, $[X_\ze,\g_\un]\subset\CC Y_\un$. Moreover, $B(X_\un,[\g_\ze,\g_\ze]=0$ implies that $[X_\un,\g_\ze]=\{0\}$. If $[X_\ze,Y_\un]=\alpha Y_\un$ then $\alpha = B(Y_\ze,[X_\ze,Y_\un]) = B([Y_\ze,X_\ze],Y_\un) = -1$. Similarly, one obtains $[Y_\ze,Y_\un]=X_\un$.

Now, we continue considering the Lie bracket on $\g_\un\times \g_\un$ as follows. Assume that $[X_\un,X_\un]= \beta_1 X_\ze+\gamma_1 Y_\ze$, By the Jacobi identity and by $[X_\un,\g_\ze]=\{0\}$, $[X_0,[X_\un,X_\un]] = 0$ and $[Y_0,[X_\un,X_\un]] = 0$. Therefore, $\beta_1=\gamma_1 = 0$. Similarly, $[X_\un,Y_\un]= [X_\un,Y_\un]]=0$.

Finally, we obtain the nonzero Lie bracket on $\g$ when $\g_\ze$ is non-Abelian that $[X_\ze,Y_\ze]=Y_\ze$, $[X_\ze,Y_\un]=-Y_\un$  and $[Y_\ze,Y_\un]=X_\un$. It is easy to see that $\g$ is the diamond Lie algebra $\g_4$.

\subsection{Solvable odd quadratic Lie superalgebras of dimension 6}
\hspace{0.8cm}Let $\g$ be a solvable odd quadratic Lie superalgebra of dimension 6, where $\g_\ze={\spa}\{X_\ze,Y_\ze,Y_\ze\}$ and $\g_\un={\spa}\{X_\un,Y_\un,Z_\un\}$ such that $B(X_\ze,X_\un)=B(Y_\ze,Y_\un)=B(Z_\ze,Z_\un) = 1$, the others are zero. If $\g_\ze$ is Abelian then $[\g_\ze,\g_\un]=\{0\}$. In this case, $\g$ can be seen as a commutative algebra and the classification can be reduced to the classification of ternary cubic forms \cite{DU10}. We will consider the case of $\g_\ze$ non-Abelian. Recall here three complex solvable Lie algebras of dimension 3 as follows:

\begin{itemize}
	\item $\g_{3,1}:\ [X_\ze,Y_\ze] = Z_\ze$,
	\item $\g_{3,2}:\ [X_\ze,Y_\ze] = Y_\ze,\ [X_\ze,Z_\ze] = Y_\ze+Z_\ze$,
	\item $\g_{3,3}:\ [X_\ze,Y_\ze] = Y_\ze,\ [X_\ze,Z_\ze] = \mu Z_\ze$, where $\left|\mu\right|\leq 1$.
\end{itemize}

\subsubsection{Case 1: $\g_\ze=\g_{3,1}$.}
Since $[Z_\ze,\g_\ze]=\{0\}$, one has $[Z_\ze,\g_\un]=\{0\}$ and $[\g_\ze,\g_\un]\subset \CC X_\un\oplus\CC Y_\un$ by the invariance of the bilinear form $B$. Moreover, $B([\g_\ze,\g_\ze],\CC X_\un\oplus\CC Y_\un) = 0$ implies that $[\g_\ze,\CC X_\un\oplus\CC Y_\un]=\{0\}$. As a consequence of the Jacobi identity, one has
$$ [\CC X_\un\oplus\CC Y_\un,\CC X_\un\oplus\CC Y_\un]\subset \Zs(\g_\ze) = \CC Z_\ze. $$

By a straightforward computation, $[Y_\ze,Z_\un] = X_\un$ and $[X_\ze,Z_\un] = -Y_\un$. Now, we assume that $[ X_\un, X_\un] = \alpha Z_\ze$,  $[ X_\un, Y_\un] = \beta Z_\ze$, $[ Y_\un, Y_\un] = \gamma Z_\ze$ and $ [ Z_\un, Z_\un] = xX_\ze+yY_\ze+zZ_\ze$. By the Jacobi identity, one has $[X_\ze,[ Z_\un, Z_\un]] + [Z_\un,[Z_\un,X_\ze]] -[Z_\un,[X_\ze,Z_\un]] = 0$.
Therefore, $[ Z_\un, Y_\un]=-\frac{y}{2}Z_\ze$. Similarly, $[ Z_\un, X_\un]=-\frac{x}{2}Z_\ze$. By the invariance of $B$, we obtain $\alpha=\beta=\gamma = 0$ and $x=y=0$.

It results that $\ [X_\ze,Y_\ze] = Z_\ze$, $[Y_\ze,Z_\un] = X_\un$, $[X_\ze,Z_\un] = -Y_\un$ and $[ Z_\un, Z_\un] = \lambda Z_\ze$.
\begin{rem} If $\lambda=0$ then the Lie superalgebra $\g$ is the six-dimensional quadratic Lie algebra $\g_{6,1}$ given in Proposition \ref{prop1.19}. It is easy to check that if $\lambda\neq 0$ then $\g(\lambda)$ is isomorphic to $\g(1)$ by the following isomorphism $A$:
$$A(X_\ze)= X_\ze,\ A\left(\frac{Y_\ze}{\lambda}\right)=Y_\ze, \ A\left(\frac{Z_\ze}{\lambda}\right)=Z_\ze, \ A\left(\frac{X_\un}{\lambda^2}\right)=X_\un, \ A\left(\frac{Y_\un}{\lambda}\right)=Y_\un, \ A\left(\frac{Z_\un}{\lambda}\right)=Z_\un.$$

However, we have a stronger result that $\g(\lambda)$ is i-isomorphic to $\g(\lambda')$ by the i-isomorphism
defined by $C(X_\ze)= aX_\ze+Y_\ze$, $C(Y_\ze)=aX_\ze+2Y_\ze$, $C(Z_\ze)=aZ_\ze$, $C(X_\un)=\frac{2}{a}X_\un-Y_\un$, $C(Y_\un)=-\frac{1}{a}X_\un+Y_\un$ and $C(Z_\un)=\frac{1}{a}Z_\un$ where $a=\left({\frac{\lambda'}{\lambda}}\right)^{1/3}$.
\end{rem}

\subsubsection{Case 2: $\g_\ze=\g_{3,2}$.}
By the invariance of $B$, one has $[\g_\ze,X_\un] = \{0\}$, $[X_\ze,Y_\un]=-Y_\un-Z_\un$, $[Y_\ze,Y_\un]=X_\un$, $[Z_\ze,Y_\un]=X_\un$, $[X_\ze,Z_\un]=-Z_\un$, $[Y_\ze,Z_\un]=0$ and $[Z_\ze,Z_\un]=X_\un$. 

By the Jacobi identity and $[\g_\ze,X_\un] = \{0\}$, one has $[X_\un,X_\un]=0$. Assume that $[X_\un,Y_\un]=aX_\ze+bY_\ze+cZ_\ze$. By a straightforward computation, $[Y_\ze,[X_\un,Y_\un]] + [X_\un,[Y_\un,Y_\ze]] - [Y_\un,[Y_\ze,X_\un]] = 0$ implies $a=0$ and $[X_\ze,[X_\un,Y_\un]] + [X_\un,[Y_\un,X_\ze]] - [Y_\un,[X_\ze,X_\un]] = 0$ implies $[X_\un,Z_\un]=-(2b+c)Y_\ze-2cZ_\ze$. As a consequence, $[X_\ze,[X_\un,Z_\un]] + [X_\un,[Z_\un,X_\ze]] - [Z_\un,[X_\ze,X_\un]] = 0$ reduces to $-4(b+c)Y_\ze-4cZ_\ze=0$. Thus, $b=c=0$.

By a similar way using the Jacobi identity, one has 
$[\g_\un,\g_\un]=\{0\}$.

\subsubsection{Case 3: $\g_\ze=\g_{3,3}$.}

Since $B([\g_\ze,\g_\ze],X_\un) = 0$, one has $[\g_\ze,X_\un]=\{0\}$. By the invariance of $B$, it is easy to obtain the nonzero Lie brackets on $\g_\ze\times\g_\un$ as follows:
$$[X_\ze,Y_\un]=-Y_\un,\ [Y_\ze,Y_\un]=X_\un, \ [X_\ze,Z_\un]=-\mu Z_\un\ \text{and}\ [Z_\ze,Z_\un]=\mu X_\un.$$
\begin{enumerate}
	\item If $\mu=0$ then $[\g_\ze,\CC X_\un\oplus \CC Z_\un]=\{0\}$. By the Jacobi identity, $[\CC X_\un\oplus \CC Z_\un,\CC X_\un\oplus \CC Z_\un]\subset \Zs(\g_\ze) = \CC Z_\ze$. We can assume that 
	$[ X_\un, X_\un] = \alpha Z_\ze$,  $[ X_\un, Z_\un] = \beta Z_\ze$, $[ Z_\un, Z_\un] = \gamma Z_\ze$ and $ [ Y_\un, Y_\un] = xX_\ze+yY_\ze+zZ_\ze$.
	Since $[X_\ze,[Y_\un,Y_\un]] + [Y_\un,[Y_\un,X_\ze]] - [Y_\un,[X_\ze,Y_\un]] = 0$, one has $ [ Y_\un, Y_\un] = 0$. As a consequence, $[Y_\ze,[Y_\un,Y_\un]] + [Y_\un,[Y_\un,Y_\ze]] - [Y_\un,[Y_\ze,Y_\un]] = 0$ implies 
 $[X_\un,Y_\un]=0$.
Similarly, by using the Jacobi identity for $X_\ze$, $Z_\un$ and $Y_\un$, we obtain 
$[ Z_\un, Y_\un] = 0$.

Moreover, $[Y_\ze,[Y_\un,Z_\un]] + [Y_\un,[Z_\un,Y_\ze]] -[Z_\un,[Y_\ze,Y_\un]]= 0$ and $[X_\un,[Y_\ze,Y_\un]] + [Y_\ze,[Y_\un,X_\un]] + [Y_\un,[Z_\un,Y_\ze]] = 0$ imply that $\alpha=\beta = 0$. Summarily, $$[X_\ze,Y_\ze]=Y_\ze, \ [X_\ze,Y_\un]=-Y_\un,\ [Y_\ze,Y_\un]=X_\un \ \text{and}\ [ Z_\un, Z_\un] = \gamma Z_\ze.$$
In this case, $\g = \g^o_{4,3}\oplusp\g^o_2(\gamma)$ decomposable.
\item If $\mu\neq 0$. By the Jacobi identity and $[\g_\ze,X_\un]=\{0\}$, one has $[X_\un,X_\un] = 0$. Assume $ [X_\un, Y_\un] = aX_\ze+bY_\ze+cZ_\ze$ then $[X_\ze,[X_\un,Y_\un]] + [X_\un,[Y_\un,X_\ze]] - [Y_\un,[X_\ze,X_\un]] = 0$
implies $aX_\ze+2bY_\ze+c(\mu+1)Z_\ze = 0$. Therefore, $a = b = 0$ and $c(\mu+1)=0$. Also, $[Y_\ze,[Y_\un,Y_\un]] + [Y_\un,[Y_\un,Y_\ze]] -[Y_\un,[Y_\ze,Y_\un]] = 0$ reduces to $[Y_\ze,[Y_\un,Y_\un]] = 2c Z_\ze$. Combining this with $[Y_\ze,\g_\ze]=\CC Y_\ze$, we obtain $c=0$ and then $[Y_\un,Y_\un] = x Y_\ze + y Z_\ze$. Similarly, 
$[X_\ze,[Y_\un,Y_\un]] + [Y_\un,[Y_\un,X_\ze]] - [Y_\un,[X_\ze,Y_\un]] = 0$
implies $3xY_\ze+y(\mu+2)Z_\ze = 0$. Since $\left|\mu\right|\leq 1$, we get $x=y=0$.

Assume $[Z_\un,Z_\un] = u X_\ze + v Y_\ze + wZ_\ze$. Since $[X_\ze,[Z_\un,Z_\un]] + [Z_\un,[Z_\un,X_\ze]] - [Z_\un,[X_\ze,Z_\un]] = 0$,
we have $u=w=0$ and $v(1+2\mu)=0$. 

If $[Z_\un,Y_\un] = \alpha X_\ze + \beta Y_\ze + \gamma Z_\ze$ then by applying the Jacobi identity for $X_\ze$, $Y_\un$ and $Z_\un$ we obtain 
$(\mu+1)\alpha = (\mu+2)\beta = \gamma(2\mu+1) = 0 $. Therefore, $\beta = 0$ since $\left|\mu\right|\leq 1$.

By the Jacobi identity for $Y_\ze$, $Y_\un$ and $Z_\un$ we get 
$[Z_\un,X_\un]=-\alpha Y_\ze$ and then $[Z_\ze,[Z_\un,Z_\un]] + [Z_\un,[Z_\un,Z_\ze]] - [Z_\un,[Z_\ze,Z_\un]] = 0$
reduces to $\alpha = 0$. Since
$[Z_\un,[Z_\un,Y_\un]] + [Z_\un,[Y_\un,Z_\un]] + [Y_\un,[Z_\un,Z_\un]] = 0$,
one has $v = -2\gamma\mu$. Combining this with $\gamma(2\mu+1) =0$, one has $v=\gamma$. Therefore, 
$[Z_\un,Z_\un] = \gamma Y_\ze$ and $[Z_\un,Y_\un] = \gamma Z_\ze $ where $\gamma(2\mu+1)=0$. 

We have two following cases:
\begin{itemize}
	\item If $\gamma = 0$ then the nonzero Lie bracket is defined $\g$ by $[X_\ze,Y_\ze] = Y_\ze$, $[X_\ze,Z_\ze] = \mu Z_\ze$, $[X_\ze,Y_\un]=-Y_\un$, $[Y_\ze,Y_\un]=X_\un$, $[X_\ze,Z_\un]=-\mu Z_\un$ and $[Z_\ze,Z_\un]=\mu X_\un$.
	\item If $\gamma\neq0$ then $\mu=-\frac{1}{2}$ and $[X_\ze,Y_\ze] = Y_\ze$, $[X_\ze,Z_\ze] = -\frac{1}{2} Z_\ze$, $[X_\ze,Y_\un]=-Y_\un$, $[Y_\ze,Y_\un]=X_\un$, $[X_\ze,Z_\un]=\frac{1}{2} Z_\un$, $[Z_\ze,Z_\un]=-\frac{1}{2} X_\un$, $[Z_\un,Z_\un] = \gamma Y_\ze, \ [Z_\un,Y_\un] = \gamma Z_\ze$. Replacing $Y_\ze$ by $\gamma Y_\ze$ and $Y_\un$ by $\gamma^{-1}Y_\un$, we obtain the result that $[X_\ze,Y_\ze] = Y_\ze$, $[X_\ze,Z_\ze] = -\frac{1}{2} Z_\ze$, $[X_\ze,Y_\un]=-Y_\un$, $[Y_\ze,Y_\un]=X_\un$,	$[X_\ze,Z_\un]=\frac{1}{2} Z_\un$, $[Z_\ze,Z_\un]=-\frac{1}{2} X_\un$, $[Z_\un,Z_\un] = Y_\ze$ and $[Z_\un,Y_\un] = Z_\ze$.
\end{itemize}
\end{enumerate}

Summarily, solvable odd quadratic Lie superalgebras of dimension 6 are listed in Table \ref{table1}. Note that the Lie superalgebras $\g_{6,i}^{o}$, $2\leq i\leq 7$, are indecomposable.
\begin{table}[ht]
		\begin{tabular}{|l|l|l|}\hline
			\textbf{$\g$} & \textbf{$\g_\ze$} & Other nonzero brackets \\ \hline
			$\g_{6,0}^{o}$ & Abelian & \\ \hline
			$\g_{6,1}^{o}$ & Abelian & $[\g_\un,\g_\un] \subset \g_\ze$ \\ \hline
			$\g_{6,2}^{o}$ & $[X_\ze,Y_\ze] = Z_\ze$ & $[Y_\ze,Z_\un] = X_\un$,\hspace{0.5cm} $[X_\ze,Z_\un] = -Y_\un$ \\ \hline
			$\g_{6,3}^{o}$ & $[X_\ze,Y_\ze] = Z_\ze$ & $[Y_\ze,Z_\un] = X_\un$,\hspace{0.5cm} $[X_\ze,Z_\un] = -Y_\un$,\hspace{0.5cm} $[ Z_\un, Z_\un] = Z_\ze$ \\ \hline
			$\g_{6,4}^{o}$ & $\ [X_\ze,Y_\ze] = Y_\ze$,  & $[X_\ze,Y_\un]=-Y_\un-Z_\un$,\hspace{0.5cm} $[Y_\ze,Y_\un]=X_\un$, \\
			 & $[X_\ze,Z_\ze] = Y_\ze+Z_\ze$ & $[Z_\ze,Y_\un]=X_\un$, \hspace{0.5cm} $[X_\ze,Z_\un]=-Z_\un$, \hspace{0.5cm} $[Z_\ze,Z_\un]=X_\un$\\ \hline
			$\g_{6,5}^{o}$ & $[X_\ze,Y_\ze] = Y_\ze$ & $\g = \g^o_{4,3}\oplusp\g^o_2(\gamma)$\\ \hline
			 
			$\g_{6,6}^{o}$  & $[X_\ze,Y_\ze] = Y_\ze,$ & $[X_\ze,Y_\un]=-Y_\un,\hspace{0.5cm}[Y_\ze,Y_\un]=X_\un, \hspace{0.5cm} [X_\ze,Z_\un]=-\mu Z_\un,$\\
			 & $[X_\ze,Z_\ze] = \mu Z_\ze$ & $\ [Z_\ze,Z_\un]=\mu X_\un,\ (\mu\neq 0,\ \left|\mu\right|\leq 1)$\\ \hline
			 
			$\g_{6,7}^{o}$ & $[X_\ze,Y_\ze] = Y_\ze,$ & $[X_\ze,Y_\un]=-Y_\un,\hspace{0.5cm} [Y_\ze,Y_\un]=X_\un, \hspace{0.5cm} [X_\ze,Z_\un]=\frac{1}{2} Z_\un,$ \\
			& $[X_\ze,Z_\ze] = -\frac{1}{2}Z_\ze $& $[Z_\ze,Z_\un]=-\frac{1}{2} X_\un,\hspace{0.5cm} [Z_\un,Z_\un] = Y_\ze,\hspace{0.5cm} [Z_\un,Y_\un] = Z_\ze.$\\ \hline
			
		\end{tabular}
		\vspace{0.5cm}
		\caption{\textbf{Complex solvable six-dimensional odd quadratic Lie superalgebras}.}\label{table1}
\end{table}
\begin{aknow}
  I am very grateful to D. Arnal, R. Ushirobira and S. Benayadi for their friendly encouragements and help. 
	
	This work is financially by the Foundation for Science and Technology Project of Vietnam Ministry of Education and Training.
 
\end{aknow}

\bibliographystyle{unsrt} 

\end{document}